\documentclass[11pt, reqno]{amsart}
\usepackage{amssymb,amscd,amsthm,amsxtra}
\usepackage{latexsym}

\usepackage{graphicx}  \usepackage{epstopdf}

\usepackage{color}
\usepackage[usenames,dvipsnames,svgnames,table]{xcolor}
\bibliography{refs}

\usepackage{hyperref}
\hypersetup{colorlinks = false}

\usepackage[mathscr]{euscript}
\renewcommand{\mathcal}[1]{{\mathscr#1}}

\newtheorem{theorem}{Theorem}[section]

\newtheorem{lemma}[theorem]{Lemma}
\newtheorem{prop}[theorem]{Proposition}
\newtheorem*{theorem*}{Theorem}
\newtheorem*{propLuZhu*}{Corollary 1 of \cite{LuZhu}}

\theoremstyle{definition}

\theoremstyle{remark}
\newtheorem{remark}[theorem]{Remark}
\numberwithin{equation}{section}

\newcommand{\R}{{\mathbb R}}

\renewcommand{\leq}{\leqslant}
\renewcommand{\le}{\leqslant}
\renewcommand{\geq}{\geqslant}
\renewcommand{\ge}{\geqslant}

\renewcommand{\epsilon}{\varepsilon }

\newlength{\defbaselineskip}
\newcommand{\setlinespacing}[1]
           {\setlength{\baselineskip}{#1 \defbaselineskip}}

\author[B. Barrios]{B. Barrios}
	\address{B. Barrios \hfill\break\indent
		Departamento de An\'{a}lisis Matem\'{a}tico,
		Universidad de La Laguna\hfill 
		\break \indent C/. Astrof\'{\i}sico Francisco S\'{a}nchez s/n, 
		38200 -- La Laguna, SPAIN}
	\email{bbarrios@ull.es}

\author[A. Quaas]{A. Quaas}
	\address{A. Quaas\hfill\break\indent
		Departamento de Matem\'{a}tica, \hfill\break\indent
		Universidad T\'ecnica Federico Santa Mar\'{\i}a
		\hfill\break\indent  Casilla V-110, Avda. Espa\~na, 
		1680 -- Valpara\'{\i}so, CHILE.}
	\email{{\tt alexander.quaas@usm.cl}}

\begin{document}

\subjclass[2010]{35B50, 35B53, 35S15, 47G20}

\keywords{Non local operators, fractional Laplacian, H\'enon equation, Moving spheres}

\thanks{}

\title[Nonlocal H\'enon problem]{The sharp exponent in the study of the nonlocal H\'enon equation in $\mathbb{R}^{N}$. A Liouville theorem and an existence result}

\begin{abstract}
We consider the nonlocal H\'enon equation 
$$(-\Delta)^s u= |x|^{\alpha} u^{p},\quad \mathbb{R}^{N},$$
where $(-\Delta)^s$ is the fractional Laplacian operator with $0<s<1$, $-2s<\alpha$, $p>1$ and $N>2s$. We prove a nonexistence result for positive solutions in the optimal range of the nonlinearity, that is, when
$$1<p<p^*_{\alpha, s}:=\frac{N+2\alpha+2s}{N-2s}.$$
Moreover, we prove that a bubble solution, that is a  fast decay positive radially symmetric solutions, exists when $p=p_{\alpha, s}^{*}$.
\end{abstract}

\maketitle


\section{Introduction}

Problems with non local diffusion that involve the fractional Laplacian operator, and other integro-differential operators, have been intensively studied in the last years since they appear when we try to model different physical situations as anomalous diffusion and quasi-geostrophic flows, turbulence and
water waves, molecular dynamics and relativistic quantum mechanics of stars
(see \cite{BoG,Co} and references). They also appear in mathematical finance (cf. \cite{A,CoT}), elasticity problems \cite{signorini}, obstacle problems \cite{BFR15,BFR16}, phase transition \cite{AB98} and crystal dislocation \cite{dfv, toland} among others.
\\

The aim of this work is to study the existence of positive solutions of the nonlocal H\'enon equation 
\begin{equation}\label{Henon} 
(-\Delta)^s u= |x|^{\alpha} u^{p},\quad \mathbb{R}^{N},\, N>2s,
\end{equation}
Here we will assume $\alpha>-2s$. In the case $\alpha\leq-2s$ it can be proved that there are not solutions of \eqref{Henon} (see Remark \ref{alal} ii) below).

The operator $(-\Delta)^s$, $0<s<1$, is the well-known \emph{fractional laplacian}, which
is defined on smooth functions as
\begin{equation}\label{operador}
(-\Delta)^s u(x) =  a_{N,s}\int_{\R^N} \frac{u(x)-u(y)}{|x-y|^{N+2s}} dy,
\end{equation}
where $a_{N,s}$ is a normalization constant that is usually omitted for brevity. The integral in \eqref{operador} 
has to be understood in the principal value sense, that is, as the limit as $\epsilon\to 0$ of the 
same integral taken in {$\R^N\setminus B_\epsilon(x)$, i.e, the complementary of the ball of center $x$ and radius $\epsilon$}.  It is clear that the fractional laplacian operator is well defined for functions that belong, for instance, to $\mathcal{L}_{2s}\cap \mathcal{C}^{1,1}_{loc}$ where
$$\mathcal{L}_{2s}:=\{u:\mathbb{R}^{N}\to \mathbb{R}:\, \int_{\mathbb{R}^{N}}\frac{|u(x)|}{1+|x|^{N+2s}}<\infty\}.$$
 {Throughout all the work we will called this solutions {\it strong or pointwise}. Of course another kind of solutions can be considered. The weaker one are the {\it distributional} solutions $u\in \mathcal{L}_{2s}$ introduced in \cite{PhS} where the equation is satisfied using the duality product
$$\langle (-\Delta)^{s}u,\varphi\rangle =\int_{\mathbb{R}^{N}} u(-\Delta)^{s}\varphi\quad \varphi\in\mathcal{S}.$$ We will specified along the work for which kind of solutions our results can be proved (see Theorems \ref{Liouville}-\ref{existencia}).} We suggest to the reader to see, for instance, \cite{guida, PhS, Stein} to read more about the basic properties of the operator and the normalization constant.
\medskip

{In the local case, that is, when $s=1$, the H\'enon equation has been introduced in modes of astrophysics (see for instance \cite{Henon}). A  close related model was introduced by Matukuma in 1930 with a different weight function $K(|x|)$ instead of $|x|^\alpha$. The case $\alpha=0$ was studied by Emden-Folwer-Lane and, thus, in this case, \eqref{Henon} is sometimes called the Emden-Folwer-Lane equation. We notice that the solutions of the H\'enon equation, and also Emden-Folwer-Lane equation, in the critical case, have a relationship with the best constants in Sobolev-Hardy inequality (see for example \cite{Lieb}). Moreover these type of problems are also connected with the mass transport and Yamabe problem (see for instance \cite{Villani}, \cite{Yamabe} and \cite{Schoen}). Since the previous one are cornerstone problems in Analysis it is hard to give a exhaustive list of references on the topic, we just mention, for example, \cite{Aubin}, \cite{talenti}, \cite{Ni},  \cite{GS}, \cite{G} ,\cite{CGS}, \cite{Chen1}, \cite{ChenLiOu}, and \cite{Stein1}.}
 
 \medskip
 
In the nonlocal framework the H\'enon equation has been studied in, for instance, \cite{DouZhou}, \cite{Yang}, \cite{LuZhu}, \cite{wei}, \cite{sire} and \cite{li}.

\medskip

Returning with the local case  ($s=1$) and the H\'enon equation, the nonexistence of positive solutions of \eqref{Henon} when $p$ is subcritical, that is,
$$1<p<\frac{N+2\alpha+2}{N-2},$$ and $\alpha> -2$ has been established by Wang, Li and Hong in \cite{WangLiHong}. The method used in \cite{WangLiHong}, that can be also applied to more general weights $K(x)\neq |x|^{\alpha}$, use the moving spheres method and the conclusion is obtained by analyzing the properties of the eigenvalues of the Laplacian. This strategy is the one that we will follow to prove our nonexistence theorem in the nonlocal framework.

\medskip

We also mention that in the recent work \cite{Jorge}, J. Garc\'ia-Meli\'an uses the well know method of moving spheres (see for example \cite{LiZhang}) to deduce that, by proving a monotonicity property, the solutions are stables. This fact allows him to use the Liouville theorem proved in \cite{DancerDuGuo} to conclude. 
\\See also \cite{Biachi, Mitidieri, PhanSouplet} for previous results in which other kind of Liouville theorems were proved in the non-optimal range or for special dimensions. 

Regarding with the critical case $$p=\frac{N+2\alpha+2}{N-2},$$ the situation is completely different when $\alpha>0$ or $-2<\alpha<0$. In fact, for positive values of the parameter, in \cite{Gladiali} was proved that there exists positive solutions that are not radially symmetric. When $\alpha$ is negative the complete classification of all the posible solutions was also stablished in \cite{Jorge} using some partial results done previously in \cite{Dolbeault}. 
\medskip

In the nonlocal case, as far as we known, it is not known a general Liouville result in all the expected range. There are some partial results in \cite{DouZhou} where the authors stablished a non existence result when 
$\alpha>0$ and 
$$\frac{N+\alpha}{N-2s} <p <\frac{N+\alpha+2s}{N-2s}:=\hat{p}.$$ 
Beyond that the range is not optimal, in our opinion, the authors in \cite{DouZhou} have a gap in the proof of their nonexistence result \cite[Theorem 1.2]{DouZhou}. More precisely, in a technical lemma, they used a local inequality, proved in \cite{Mitidieri_local} in the local framework, that does not seems trivially true for the solutions in the nonlocal case. 
Moreover, when $p=\hat{p}$ they affirm in \cite[Theorem 1.3]{DouZhou} that if there exists a nonnegative solution then it has to be radially symmetric around the origin using the work of \cite{LuZhu}. However the result of Lu and Zhu that they cited can be applied only when $\alpha$ is negative, not when $\alpha$ is positive, and $p=p^*_{\alpha}$ (not $p=\hat{p}$), where
\begin{equation}\label{critico}
p^*_{\alpha, s}:=\frac{N+2\alpha+2s}{N-2s}\neq \hat{p}.
\end{equation}

\medskip
As occurs in the local case, we guess that the existence of solutions in the case $p=\hat{p}$ is not expected attended to the complete range $1<p<p^*_{\alpha, s}$ for which the Liouville theorem is expected (see Theorem \ref{Liouville} below). For that reason our first objective is to obtain the following
 
\begin{theorem}\label{Liouville}
If we assume that $0<s<1$, $\alpha>-2s$ and 
$$1<p<p^*_{\alpha, s},$$
then the problem \eqref{Henon} does not have any positive distributional solution.
\end{theorem}

Since to prove Theorem \ref{Liouville} we will follows, mainly, the strategy developed in \cite{WangLiHong}, that can be applied to more general weights $K(x)\neq |x|^{\alpha}$, our nonexistence result can be also proved for some more general equations than \eqref{Henon}. However, for simplicity of the exposition, we will consider the particular case $K(x)=|x|^{\alpha}$.

\begin{remark}
Just before submitting this version of the manuscript we found a very recent preprint where the conclusion of our Theorem 1.1 can be deduced (see \cite{Dai}). Nevertheless, although the techniques used by Dai and Qin are similar as the one we develop here (mainly the moving sphere by using the Kelvin transform for the fractional Laplacian) our proof is different and we consider that our arguments and steps are much more simple and direct than the used by Dai and Qin.
\end{remark}

To complement the previous nonexistence result we will study the existence of the solutions when $p=p^*_{\alpha, s}$ for some range of $\alpha$, and the decay estimates for the solutions, that were not known until the date, will be stablished. To explain the type of result that we have obtained in this context we notice that in the local case, writing the Laplacian in radial coordinates, it is very easy to check that the ``bubble" functions
\begin{equation}\label{bbb}
b(x)=C(N,\alpha)\left(\frac{1}{1+|x|^{2+\alpha}}\right)^{\frac{N-2}{2+\alpha}},
\end{equation}
are solutions of \eqref{Henon} when $s=1$. However in the nonlocal case this computation is not at all straightforward because the formula of the fractional Laplacian for radial functions has not a simple form.
Notice that in the nonlocal case when $\alpha=0$ the bubble is given by
$$b(x)=C(N,s)\left(\frac{\lambda}{{\lambda^2}+|x|^{2}}\right)^{\frac{N-2s}{2}},$$
for some $C(N,s),\, \lambda>0$, (see \cite{Lieb}).  We show up here that the computations in \cite{Lieb} cannot be adapted for the case $\alpha\neq 0$ since the Fourier transform of for functions of the form
$$u_{\mu}(x)=(1+|x|^\gamma)^{-\mu},\, {\mu>0},\, 0<\gamma\neq 2,$$
has not the desired expression as occurs when $\gamma=2$ (see Section 4 below). Thus, in order to prove the existence of the ``bubble" solution in the nonlocal case, that is, a radial solution when $p=p^*_{\alpha, s}$, we have to use an alternative approach passing through a Emden-Fowler change of variables that forces us to study an alternative problem similar as the one that appear in \cite{Manuel}. 
For that we made a one-dimensional reduction and we work with a new problem of the form
$$\mathcal{T} \bar{v}(\kappa) + \mathcal{A}_{s,N} \bar{v}(\kappa)= \bar{v}^{p^*_{\alpha,s}}(\kappa),\, \quad \kappa\in\mathbb{R}.$$
Here the operator $\mathcal{T}$ is an integral operator with a kernel whose singularity is similar to the ones of to the fractional Laplacian and that has exponential decay at infinity (see Section 4 below). For the existence of a solution of the new problem, some restriction in case $s<1/2$ appears by the fact that we need to assume that $p^*_{\alpha, s}$ is subcritical in one dimension. For the proof of qualitative properties, such as fast-decay of the bubble solution, i.e., decay as $r^{-N+2s}$; we obtain a suitable and sharp estimate for the function $\bar{v}$.

Summarizing, we have the  following second main result.

\begin{theorem}\label{existencia}
Let us consider $p=p^*_{\alpha, s}$ defined in \eqref{critico}.
\begin{itemize}
\item[a)] If 
$$\mbox{$1/2\leq s<1$, $0<\alpha$,}$$
or
$$\mbox{$0<s<1/2$, $0<\alpha<2s(N-1)/(1-2s)$,}$$
then the problem \eqref{Henon}  admits a classical positive radially symmetric solution $u\in L^{\infty}(\mathbb{R}^{N})\cap C^{\infty}(\mathbb{R}^{N})$ of \eqref{Henon}.

\item[b)] If $0<s<1$, $-2s<\alpha<0$, then there exists a weak (variational)  positive radially symmetric solution $u\in L^\infty(\R^N) \cap H^s_{loc}(\R^N)\cap C^\infty(\R^N \setminus \{0\})$ of \eqref{Henon}.

\item[c)] All radial solution $u$ of \eqref{Henon} are fast decay, that is, there exist constants $c_i>0$, $i=1,\,2$, such that 
$$c_1 r^{-N+2s}\leq u(r)\leq c_2 r^{-N+2s}.$$
\end{itemize}
\end{theorem}

To finish this introduction we give some comments respect to the previous existence theorem:
\\

i) Another existence result related with the nonlocal H\'enon equation can be found in, for instance, \cite{Yang} where the author, using some ideas of \cite{Giamp}, proves the existence of a positive, radially symmetric solution when $-2s<\alpha<0$ and $p=p^*_{\alpha, s}$ establishing a refinement of a Sobolev-Hardy inequalities in terms of a Morrey-Campanato space. In addition J. Yang showed that all positive energy solution has to be radially symmetric and strictly decreasing. We notice here that, beside that in \cite{Yang} an existence result like b) was obtained, since our method is different, based on a Emden-Fowler change of variable, and the uniqueness for equations like \eqref{Henon} is still nowadays an interesting problem, is not clear if both solutions coincide or not.
\\

ii) We expected that  
$$\lim_{r \to \infty}u(r)r^{N-2s},$$
exists and $u$ may have an explicit formula as in the local case (see \eqref{bbb}). But this is also wide open. Notice that for $p\not=2$ there is an argument for $(\alpha=0)$ to established the exact behavior of the solutions for the $p$-fractional Laplacian equation and the associated critical exponent (see \cite{Bras}). However we cannot see clearly how the method developed in  \cite{Bras} can be applied to the case $p=2$, and, what it does not seem easy at all, when $\alpha\neq0$.
\\

iii) We emphasize again that for $0<s<1/2$ and $\alpha>2s(N-1)/(1-2s)$ the problem in the Emden-Fowler variables is super-critical so is not clear if a solution with fast decay can exists (see Section 4 bellow).  


\medskip

The rest of the paper is organized as follows: in Section 2 we give we introduce the notion of solutions that will be used along the work and some relationships. Section 3 deals with the proof of Theorem \ref{Liouville}. Finally in Section 4 we obtain the proof of Theorem \ref{existencia}. 

We remark here that along the work we will denote by $C$ a positive constant that may change from line to line.

\section{Weak solution and integral form}

In this section we revise the notion of solution {and we present a key lemma that will be used in the next section to prove the Liouville's result. 
\\
As we have commented in the Introduction of the work, different kind of solutions of \eqref{Henon} can be considered:
\\-{\it Strong} (or {\it pointwise}) solutions $u\in \mathcal{L}_{2s}\cap \mathcal{C}^{1,1}_{loc},$ 
\\-{\it weak} (or {\it variational}) solutions $u\in H^{s}(\mathbb{R}^{N}),$
and
\\-{\it distributional} solutions $u\in \mathcal{L}^{2s}$.
\\As we mention in the introduction, the first objective of the work on hands is to prove a nonexistence result for distributional solutions (see Theorem \ref{Liouville}). Thus, the following auxilliary lemma, which proof will be based on ideas of \cite[section 2]{Zhuo1} (see also \cite{Zhuo}), will be developed for the weaker notion of solutions. However the conclusion it is also true for those solutions that admit maximum and comparison principles.} 

\begin{lemma} \label{weak}
If $u\in \mathcal{L}_{2s} \cap L^\infty_{loc}(\mathbb{R}^{N})$ is a nonnegative distributional solution of \eqref{Henon} then 
\begin{equation}\label{eqint}
u(x)=\int_{\mathbb{R}^{N}}\frac{|y|^{\alpha}u^p(y)}{|x-y|^{N-2s}}\, dy.
\end{equation}
In addition, if $\alpha>0$ then  $u\in C^\infty(\mathbb{R}^{N})$, if $-2s<\alpha<0$ then $u\in C^\infty(\mathbb{R}^{N}\setminus \{0\})$

\end{lemma}

\begin{remark}\label{alal} 
{i) By the previous result, if $\alpha>0$, locally bounded {\it distributional} solutions of \eqref{Henon} are {\it classical}, that is, {\it strong} (or {\it poinwise}).} Moreover when $\alpha\leq -2s$ there is not a nonnegative locally bounded solution of \eqref{eqint} because the integral of the right is divergent at zero. Notice that the non-existence for  $-2\geq\alpha$ in the case $s=1$ is base on a ODE argument, (see \cite[Theorem 2.3] {DancerDuGuo}). This argument is not valid for the nonlocal case. 
\\
ii) Close related results for the case $-2s<\alpha<0$   can be found in \cite{DouZhou}, including a different regularity results at the singularity. 
\end{remark}

\begin{proof} As we commented before we follow closely the ideas given in \cite{Zhuo1} (see also \cite{Zhuo}) that use the Green's function $G_{R}(x,y)$ for the fractional Laplacian in a ball. Since it is known (see for example \cite{Kul}  and \cite{Chen2}) that $|G_R(x,y)| \leq C|x-y|^{-(N-2s)}$, we notice that the solution of \eqref{Henon} in $B_{R}(0)$ with zero boundary condition given by
$$\int_{B_{R}(0)} G_{R}(x,y)|y|^{\alpha}u^{p}(y)\, dy,\, x\in B_{R}(0), $$
is well define. Therefore using the fact that
$$\int_{\mathbb{R}^{N}}{\frac{|y|^{\alpha}}{|x-y|^{N-2s}}\, dy}=\infty,\quad\alpha>-2s,$$
applying the same argument, base on a maximum principle, that can be found in \cite[Theorem 3]{Zhuo}, we obtain 
$$u(x)=\int_{\mathbb{R}^{N}}\frac{|y|^{\alpha}u^p(y)}{|x-y|^{N-2s}}\, dy.$$

The regularity is nowadays quite standard and it is based on a bootstrapping  argument. See for example a localized version of the regularity results of Schauder type in \cite[Theorem 2.1]{CFQ} or \cite{CS,PhS} .
\end{proof}

\section{Liouville Theorem}

The aim of this section is to prove Theorem \ref{Liouville}, that is the fact that does not exist any positive solution of the nonlocal H\'enon equation when $p$ is subcritical. For that we will follow the ideas developed in \cite{LiZhang,WangLiHong} for the local case taking advantage that the fractional Kelvin transform, that is fundamental to use the method of moving spheres, has been studied thoroughly in the nonlocal framework in  \cite{Bogdan} (see also \cite {Xavi}). First of all, using the fact that $(-\Delta)^{s}|x|^{2s-N}=0$ when $x\neq 0$, it is easy to check that
$$(-\Delta)^s \overline{u}(x)=\frac{1}{|x|^{N+2s}}(-\Delta)^{s}u\left(\frac{x}{|x|^2}\right),$$
with
$$\overline{u}(x):=\frac{1}{|x|^{N-2s}}u\left(\frac{x}{|x|^2}\right),\quad x\in\mathbb{R}^{N}\setminus\{0\}.$$
Then, for $\lambda>0$, using the homogeneity property of the fractional Laplacian we get that
\begin{equation}\label{eq-kelvin}
(-\Delta)^s u_{\lambda}(x)=\left(\frac{\lambda}{|x|}\right)^{N+2s}(-\Delta)^{s}u\left(\frac{\lambda^2x}{|x|^2}\right),
\end{equation}
where
\begin{equation}\label{kelvin}
u_{\lambda}(x):=\frac{1}{\lambda^{N-2s}}\overline{u}\left(\frac{x}{\lambda^2}\right)=\left(\frac{\lambda}{|x|}\right)^{N-2s}u\left(\frac{\lambda^2x}{|x|^2}\right),\quad x\in\mathbb{R}^{N}\setminus\{0\},
\end{equation}
is the Kelvin transform of $u$.  Let us now define
\begin{equation}\label{diferencia}
w_{\lambda}(x):=u_{\lambda}(x)-u(x),\quad x\in\mathbb{R}^{N}\setminus\{0\}.
\end{equation}
Before proving the auxiliar results needed to obtain the Liouville Theorem, for the convenience of the reader, we introduce now a maximum principle for narrow domains for our problem (see \cite{BMS,ChenLiZhang}). That is
\begin{lemma}\label{narrow}
Let $v\in\mathcal{L}_{2s}\cap \mathcal{C}^{1,1}_{loc}$ and $\lambda>0$. If $v(x)=-v_{\lambda}(x)$ for $x\in B_{\lambda}(0)\setminus\{0\}$ and satisfies
$$
\left\{
\begin{aligned}
(-\Delta)^s v-\overline{C}|x|^{\alpha}v&\geq 0 &&  x\in\Omega\subseteq B_{\lambda}\setminus\{0\},\, \overline{C}>0,\\
{v}&\geq0&& \mbox{in } (B_{\lambda}\setminus\{0\})\setminus\Omega.
\end{aligned}\right.
$$
with $\alpha>-2s$ then there exists a small $\delta_0>0$ such that $\inf_{\Omega}v\geq 0$ as long as 
$$\Omega\subseteq \{x\in\mathbb{R}^{N}: \lambda-\delta_0<|x|<\lambda\}.$$
\end{lemma}
\begin{proof}
Suppose by contradiction that $\inf_{\Omega}v =v(x_{min})<0$. Following verbatim \cite[Theorem 2.2]{ChenLiZhang} we get that
$$(-\Delta)^s \widetilde{v}(x_{min})\leq v(x_{min})F(\lambda, x_{min}),$$
where $\widetilde{v}(x):=v(x)-v(x_{min})$ and $F(\lambda, x_{min})\geq C \delta^{-2s}$. Since on the other hand 
$$(-\Delta)^s\widetilde{v}(x_{min}) -\overline{C}|x_{min}|^{\alpha}\widetilde{v}(x_{min})=(-\Delta)^s {v}(x_{min}) \geq \overline{C}|x_{min}|^{\alpha}v(x_{min}),$$
the contradiction follows choosing $\delta$ small enough using the fact that $\alpha>-2s$ and $v(x_{min})<0$.
\end{proof}
By the previous lemma we have the next
\begin{prop}
Let $u\in \mathcal{L}_{2s}\cap \mathcal{C}^{1,1}_{loc}(\mathbb{R}^{N}\setminus \{0\})$ be a nonnegative solution of \eqref{Henon} with $1<p<p^*_{\alpha,s}$. There exists $\lambda_0>0$ such that 

\begin{equation}\label{lucia}
\mbox{$w_{\lambda}(x)\geq 0$ for every $x\in B_{\lambda}\setminus \{0\}$, $\lambda<\lambda_0$,}
\end{equation}
where $w_{\lambda}$ and $p_{\alpha,s}$ were given in \eqref{diferencia} and \eqref{critico} respectively. 
\end{prop}
\begin{proof}  To prove \eqref{lucia} let $\lambda>0$ be a small parameter fixed but arbitrary. Since by Lemma \ref{weak}
$$u(x)=\int_{\mathbb{R}^{N}}\frac{|y|^{\alpha}u^p(y)}{|x-y|^{N-2s}}\, dy,$$
for $|x|\geq 2$ we get that
$$u(x)\geq \int_{B_1\setminus B_{1/2}}\frac{|y|^{\alpha}u^p(y)}{|x-y|^{N-2s}}\, dy\geq  \int_{B_1\setminus B_{1/2}}\frac{C}{|x-y|^{N-2s}}\, dy=\frac{C_0}{|x|^{N-2s}}.$$
Since the previous inequality implies that
$$u_{\lambda}(x)\geq \frac{C_0}{\lambda^{N-2s}},\quad x\in B_{\frac{\lambda^2}{2}}\setminus\{0\},$$ 
then 
\begin{equation}\label{dentro}
w_{\lambda}(x)>0,\quad x\in B_{\frac{\lambda^2}{2}}\setminus\{0\},
\end{equation}
choosing $\lambda<\lambda_1$ for some $\lambda_1$ small enough. On the other hand by \eqref{eq-kelvin} and the fact that $p$ is subcritical, we have that
\begin{eqnarray}
(-\Delta)^{s}w_{\lambda}(x)&=&|x|^{\alpha}\left(\left(\frac{\lambda}{|x|}\right)^{N+2s+2\alpha-p(N-2s)}u^p_{\lambda}(x)-u^p(x)\right)\label{calima}\\
&\geq& p\, |x|^{\alpha}\varphi(x) w_{\lambda}(x),\quad x\in B_{\lambda}\setminus\{0\},\nonumber
\end{eqnarray}
where $u_{\lambda}^{p-1}\leq\varphi\leq u^{p-1}$. Let us take $0<\lambda<\lambda_2$ that guarantee that $\lambda<\frac{\lambda^2}{2}+\delta_0$ where $\delta_0$ is given in Theorem \ref{narrow}. Since $0\leq\varphi(x)\leq C$ when $x\in B_{\lambda}$ then the Theorem \ref{narrow} can be applied with $\Omega=B_{\lambda}\setminus B_{\frac{\lambda^2}{2}}$ obtaining that
\begin{equation}\label{fuera}
w_{\lambda}(x)\geq0,\quad x\in B_{\lambda}\setminus B_{\frac{\lambda^2}{2}}.
\end{equation}
Thus, the desired conclusion follows from \eqref{dentro} and \eqref{fuera} for $\lambda_0=\min \{\lambda_1, \lambda_2\}$.
\end{proof}

By the previous proposition 
\begin{equation}\label{sup}
0<\widetilde{\lambda}:=\sup\{\mu>0:\, w_\lambda{\geq 0}\mbox{ in $B_\lambda\setminus \{0\}$ with $0<\lambda<\mu$}\},
\end{equation}
is well define and $\widetilde{\lambda}\leq\infty$. We analyze now if $\widetilde{\lambda}$ can be equal to infinite.
\begin{prop}
Let $u\in \mathcal{L}_{2s}\cap \mathcal{C}^{1,1}_{loc}(\mathbb{R}^{N}\setminus\{0\})$ be a nonnegative solution of \eqref{Henon} with $1<p<p^*_{\alpha}$. If $\widetilde{\lambda}<\infty$ then $w_{\widetilde{\lambda}}\equiv 0$
in $B_{\widetilde{\lambda}}\setminus \{0\}.$
\end{prop}
\begin{proof} Here we follows the classical moving type argument (see \cite{ChenLiZhang}).
First of all we notice that, by continuity, $w_{\widetilde{\lambda}}\geq 0$ in $B_{\widetilde{\lambda}}\setminus \{0\}$.
Thus, by the strong maximum principle, $w_{\widetilde{\lambda}}>0$  or  $w_{\widetilde{\lambda}}\equiv 0$
in $B_{\widetilde{\lambda}}\setminus \{0\}$. We assume by contradiction that the first possibility holds.

Let $0<\delta<\delta_0$ be a small parameter such that 
$$0<m:=\min_{B_{\widetilde{\lambda}-\delta}\setminus \{0\}} w_{\widetilde{\lambda}},$$
where $\delta_0$ was given in Lemma \ref{narrow}. Then, using the continuity of $w_\lambda$ with respect to $\lambda$, we get that 
$$\mbox{$0<m/2<w_\lambda$ in $B_{\widetilde{\lambda}-\delta}\setminus\{0\}$ if {$\lambda\in (\widetilde{\lambda}, \widetilde{\lambda}+(\delta_0-\delta))$}}.$$ 
Using now the Lemma \ref{narrow} with $\Omega=B_{\lambda}\setminus B_{\widetilde{\lambda}-\delta}$ to get $w_\lambda\geq 0$ in $B_{\lambda}\setminus\{0\}$ for $\lambda\in (\widetilde{\lambda}, \widetilde{\lambda}+(\delta_0-\delta))$ that contradicts the definition of $\widetilde{\lambda}>0$.
\end{proof}

The previous proposition implies that $\widetilde{\lambda}$ has to be infinite because, using \eqref{calima}, we know that $u_{\widetilde{\lambda}}(x)= u(x)$ can not be possible if $p<p^{*}_{\alpha,s}$. Thus our last step will be to reject the possibility that  $\widetilde{\lambda}=\infty$ as the next result shows. Observe that once we rule this option the nonexistence of a positive solution of \eqref{Henon} and the proof of Theorem \ref{Liouville} will be done.
\begin{prop}
Let $\widetilde{\lambda}$ be defined in \eqref{sup}. Then $\widetilde{\lambda}\neq\infty$.
\end{prop}
\begin{proof}
Following the ideas done in \cite{Lin} for the local case let us suppose by contradiction that $\widetilde{\lambda}=\infty$. It is clear that \eqref{kelvin} and \eqref{sup} imply 
$$\mbox{$u_{\lambda}(x)\leq u(x)$ if $|x|\geq\lambda$ with $0<\lambda<\tilde{\lambda}$.}$$
Then, for every $|x|\geq 1$, we can consider $\lambda:=|x|^{1/2}$ obtaining that
$$u(x)\geq u_{|x|^{1/2}}(x)\geq \frac{1}{|x|^{\frac{N-2s}{2}}}\min_{|z|\leq1} u(z):= \frac{c_0}{|x|^{\frac{N-2s}{2}}}.$$
Since $1<p<p^*_{\alpha}$ this clearly implies that
\begin{equation}\label{infty}
\lim_{|x|\to \infty} |x|^{\alpha+2s} u^{p-1}(x)=\infty.
\end{equation}
Thus the function $v(z):=u(x+|x|z)$,  $x\in\mathbb{R}^{N}$, satisfies
$$
\left\{
\begin{aligned}
(-\Delta)^s v(z)= |x|^{2s}c(x+|x|z)v(z),&&|z|\leq 1.\\
v\geq 0,&&\mathbb{R}^{N},
\end{aligned}\right.
$$
where $c(x):= |x|^{\alpha} u^{p-1}(x)$. Therefore by \cite[Theorema 1]{Ariel} we get that
$$\lambda_1(B_1)>\sup_{|z|\leq 1}|x|^{2s}c(x+|x|z),\, x\in\mathbb{R}^{N},$$
where $\lambda_{1}(B_{1})$ the first eigenvalue of the fractional Laplacian. The previous inequality clearly implies a contradiction with \eqref{infty} because $|x|$ can be taken arbitrarily big. Thus the desired conclusion follows.
\end{proof}

\section{Critical exponent and existence of the bubble}
The aim of this section will be to prove the existence result announced in Theorem \ref{existencia}, that is, establish the existence of radially symmetric solutions of the nonlocal critical H\'enon equation and the fast decay of them. 

Before introducing an aditional functional setting that will be use along this section to get the previous two objective, we summarize briefly how the form of the, commonly called, ``bubble" can be obtain when $\alpha=0$. As we have commented in the introduction of the present work the strategy followed in the case $\alpha\neq 0$ and $s=1$ can not be applied in the nonlocal framework because the expression of the fractional Laplacian for radial functions it can not be simplified as occurs with the Laplacian. However when $s\neq 1$ but $\alpha=0$ we can still do an explicit computation to get the form of the solution (see \cite{Lieb}). The key point when $\alpha=0$ is that the Fourier transform of the radial function 
$$u_{\mu}(x)=(1+|x|^2)^{-\mu},\, \mu>0,$$
is well known, and is given by 
$$\mathcal{F}{u_{\mu}}(\xi)=|\xi|^{\mu-\frac{N}{2}}K_{\mu-\frac{N}{2}}(|\xi|),$$
where $K$ is a Bessel function that satisfies $K_{a}=(-1)^{a}K_{-a}$, $a\in\mathbb{R}$ (see \cite{Lieb}). Using this property of $K$ it follows that
$$|\xi|^{-2s}\mathcal{F}{u_{\frac{N-2s}{2}}}(\xi)=|\xi|^s K_{-s}(|\xi|)= C|\xi|^{s}K_s(|\xi|)= C \mathcal{F}{u_{\frac{N+2s}{2}}}(\xi),$$
that is,
$$|x|^{-(N+2s)}\ast u_{\frac{N-2s}{2}}(x) = C u_{\frac{N+2s}{2}}(x)= C \left(u_{\frac{N-2s}{2}}\right)^{\frac{N+2s}{N-2s}}(x).$$
This gives that all radially symmetric solutions when $\alpha=0$ must be the standard bubbles in the nonlocal case given by
$$b(x)=C(N,s)\left(\frac{\lambda}{{\lambda^2}+|x|^{2}}\right)^{\frac{N-2s}{2}},$$
for some $C(N,s),\, \lambda>0.$ The computations done just before show up that the strategy done when $\alpha=0$ and $s\neq 1$ neither can be adapted to our case because the Fourier transform of 
$$u_{\mu}(x)=(1+|x|^\gamma)^{-\mu},\, \mu>0,\, 0<\gamma\neq 2,$$
has not the desired expression as occurs when $\gamma=2$. Therefore to find the shape of the bubble when $0<s<1$, $\alpha\neq 0$, we have to use a complete different strategy as we will expose then. 
\subsection{An alternative problem}
Following the ideas done in \cite{Manuel} we will apply the Emden Fowler change of variables in \eqref{Henon} when $p=p^*_{\alpha,s}$, $\alpha>-2s$. That is we look for radial solutions of the form
\begin{equation}\label{dobleuve}
w(r)=r^{-\frac{N-2s}{2}} v(r),
\end{equation}
where $v$ is some radial function to be determinate. Denoting by 
\begin{equation}\label{beta}
\beta:=-\frac{N-2s}{2},
\end{equation}
 It is clear that
\begin{eqnarray*}
(-\Delta)^{s} w(r)&=&\int_{0}^{\infty}\int_{\mathbb{S}^{N-1}}\frac{r^{\beta}v(r)-\rho^{\beta}v(\rho)}{|r^2+\rho^2-2r\rho\langle\theta, \sigma\rangle|^{\frac{N+2s}{2}}}\, \rho^{N-1}\, d\sigma d\rho\\
&=&r^{-2s+\beta}\int_{0}^{\infty}\int_{\mathbb{S}^{N-1}}\widetilde{\rho}^{\, N-1}\frac{v(r)-\widetilde{\rho}^{\, \beta}v(r\widetilde{\rho})}{|1+\widetilde{\rho}^{\, 2}-2\widetilde{\rho}\langle\theta, \sigma\rangle|^{\frac{N+2s}{2}}}\, d\sigma d\widetilde{\rho},
\end{eqnarray*}
where $\widetilde{\rho}=\rho/r$. Then
\begin{equation}\label{eq-dobleuve}
(-\Delta)^{s} w(r)=r^{-2s+\beta}(\mathcal{L} v(r) + \mathcal{A}_{s,N} v(r)),\quad r>0,
\end{equation}
with
\begin{equation}\label{ele}
\mathcal{L} v(r):= \int_{0}^{\infty}\int_{\mathbb{S}^{N-1}}\widetilde{\rho}^{\, N-1+\beta}\frac{v(r)-v(r\widetilde{\rho})}{|1+\widetilde{\rho}^{\, 2}-2\widetilde{\rho}\langle\theta, \sigma\rangle|^{\frac{N+2s}{2}}}\, d\sigma d\widetilde{\rho},
\end{equation}
and
\begin{equation}\label{a}
\mathcal{A}_{s,N}:=\int_{0}^{\infty}\int_{\mathbb{S}^{N-1}}\widetilde{\rho}^{\, N-1}\frac{1-\widetilde{\rho}^{\, \beta}}{|1+\widetilde{\rho}^{\, 2}-2\widetilde{\rho}\langle\theta, \sigma\rangle|^{\frac{N+2s}{2}}}\, d\sigma d\widetilde{\rho}.
\end{equation}
Thus from \eqref{dobleuve}, \eqref{eq-dobleuve} and the fact that $\alpha+\beta p^*_{\alpha,s}=-2s+\beta,$ we conclude that
\begin{equation}\label{eq-uve}
\mathcal{L} v(r) + \mathcal{A}_{s,N} v(r)= v^{p^*_{\alpha,s}}(r),\, \quad r>0.
\end{equation}
Finally doing the Emden Fowler change of variable $r=e^{\kappa}$, $\widetilde{\rho}=e^{\tau-\kappa}$ the function $\bar{v}(\kappa)=v(e^{\kappa})$, $\kappa\in\mathbb{R}$, satisfies
$$\mathcal{T} \bar{v}(\kappa) + \mathcal{A}_{s,N} \bar{v}(\kappa)= \bar{v}^{p^*_{\alpha,s}}(\kappa),\, \quad \kappa\in\mathbb{R},$$
where
\begin{eqnarray}
\mathcal{T}\bar{v}(\kappa)&=&\int_{\mathbb{R}}\int_{\mathbb{S}^{N-1}}e^{(\tau-\kappa)(N+\beta)}\frac{v(e^{\kappa})-v(e^{\tau})}{|1+e^{-2(\kappa-\tau)}-2e^{-(\kappa-\tau)}\langle\theta, \sigma\rangle|^{\frac{N+2s}{2}}}\, d\sigma d\tau.\nonumber\\
&=&\int_{\mathbb{R}}(\bar{v}(\kappa)-\bar{v}(\tau))K(\kappa-\tau)\, d\tau\label{laT},
\end{eqnarray}
being
\begin{equation}\label{kernel}
K(t)=e^{-t\frac{N+2s}{2}}\int_{\mathbb{S}^{N-1}}\frac{1}{|1+e^{-2t}-2e^{-t}\langle\theta, \sigma\rangle|^{\frac{N+2s}{2}}}\, d\sigma,\quad t\in\mathbb{R}.
\end{equation}
We notice here that, in fact, the operators $\mathcal{L}$, $\mathcal{T}$ and also the constant $\mathcal{A}_{s,N}$ depend of $\beta$, but, since along all this section, this $\beta$ is fixed and defined in \eqref{beta}, for simplicity, we omit it. We show up now some useful properties, proved in \cite{Manuel} (see also Remark \ref{extranew} below), that we will use later
\begin{prop}\label{prop-kernel}
It is true that
\begin{itemize}
\item the constant $\mathcal{A}_{s,N}$ given \eqref{a} is positive,
\item $\displaystyle K(t)= c(N, s)\int_{0}^{\pi}\sin y^{N-2s}(\cosh\, t- \cos y)^{-\frac{N+2s}{2}}\, dy, $ for some $c(N, s)>0$,
\item the kernel $K$ is even, strictly positive and satisfies
\begin{equation}\label{importante}
K(t)\sim \frac{1}{t^{1+2s}},\, t\to 0,\quad K(t)\sim e^{-t\frac{N+2s}{2}},\, t\to\infty.
\end{equation}
That is, around the singularity of the origin, the kernel behaves like the ones of the fractional Laplacian operator in dimension one.
\end{itemize}
\end{prop}
\subsection{Existence of solutions.}
Our objective now is to obtain the existence of solutions of \eqref{eq-uve-buena} that will imply the existence of radially symmetric solutions of \eqref{Henon} when $p=p^{*}_{\alpha,s}$. Observe that, to be consistence with \cite[Theorem 1]{LuZhu} (and \cite{Yang}), the existence have to be proved, at least, for every $-2s<\alpha<0$. For that we will introduce the functional framework needed to work with. First of all we consider the Sobolev space
$$H^s_{K}(\mathbb{R}):=\left\{u:\mathbb{R}\to\mathbb{R}:\, u\in L^{2}(\mathbb{R})\mbox{ and } [u]_{H^s_{K}(\mathbb{R})}<\infty\right\},$$
where
$$[u]_{H^s_{K}(\mathbb{R}^{N})}:=\int_{\mathbb{R}}\int_{\mathbb{R}}(u(\kappa)-u(\tau))^2K(\kappa-\tau)\, d\tau\, d\kappa,
$$
with $K$ given in \eqref{kernel}. By Proposition \ref{prop-kernel} we deduce that $K(t)$ is an even, decreasing monotone function that satisfies $K\in L^{1}(\mathbb{R},\, \min \{|x|^2,1\}\, dx)$. Moreover it can be checked that
$$\sup_{\widetilde{s}\, \geq 0}\left\{\lim_{r\to 0} r^{2\, \widetilde{s}}\int_{B_1(0)\setminus B_{r}(0)} K(t)\, dt=\infty\right\}=s,$$
and
$$\varliminf_{r\to 0} r^{2s}\int_{B_1(0)\setminus B_{r}(0)} K(t)\, dt=\frac{1}{2s}>0.$$
Thus, by \cite[Theorem 3.1]{ChenHajaiej}  and a density argument we get the next
\begin{prop} (\it Sobolev inequalities) \label{sobolev-ineq}
Let $0<s<1/2$. If $u\in H^s_{K}(\mathbb{R})$ then
\begin{equation}\label{hhss}
\|u\|^2_{L^{2^{*}_{s}}(\mathbb{R})}\leq c [u]_{H^s_{K}(\mathbb{R})},
\end{equation}
for some $c>0$. Therefore the space $H^s_{K}(\mathbb{R})$ is continuously embedded in $L^q(\mathbb{R})$ for every $1\leq q\leq 2/(1-2s):=2^{*}_{s}$. 
\end{prop}
\begin{proof} First of all we notice that in \cite[Theorem 3.1]{ChenHajaiej}  the authors obtained \eqref{hhss}
when $u\in H^s_{K}(\mathbb{R})$ has compact support. Thus, adapting to our setting \cite[Lemma 2.2 and Lemma 2.3]{Fis}, since we get that $\mathcal{C}_{0}^{\infty}(\mathbb{R})$ is dense in $H^s_{K}(\mathbb{R})$, by a standard density argument, we conclude \eqref{hhss} is true for all $u\in H^s_{K}(\mathbb{R})$, as wanted.
\end{proof}

Further, by \cite[Theorem 3.2]{ChenHajaiej}, we get
\begin{prop} (\it Compact Embeddings) \label{compact}
Let $0<s<1/2$, $\Omega\subseteq\mathbb{R}$ a bounded domain and $1\leq q< 2/(1-2s)$. Then every bounded sequence in $H^s_{K}(\mathbb{R})$ has a convergent subsequence in $L^{q}(\Omega)$. That is, the embedding $H^s_{K}(\mathbb{R}) \hookrightarrow L_{loc}^q(\mathbb{R})$ is compact.
\end{prop}
We highlight the important fact that $p^*_{\alpha,s}$, given in \eqref{critico}, is subcritical for the Sobolev space $H^{s}_{K}(\mathbb{R})$ in the sense that
$$p^*_{\alpha,s}=\frac{N+2s+2\alpha}{N-2s}<\frac{1+2s}{1-2s}= 2^{*}_{s}-1,\mbox{when $s<1/2$,}$$ 
for every $\alpha<0$ and with some restrictions in the case that $\alpha>0$. In fact, since $N>2s>1$,
$$2s+\alpha(1-2s)<2s<2sN,\mbox{ when $\alpha<0$}.$$
In the case $\alpha>0$ we have to add the hypothesis
$$\alpha<\frac{2s(N-1)}{1-2s}\left(\stackrel{ s\to 1/2^{-}}{\longrightarrow}\infty\right).$$
There is nothing to prove in the case $s\geq1/2$ because there is not critical Sobolev exponent in dimension one. We also notice that $p^*_{\alpha,s}>1$ because $-2s<\alpha$. Therefore, by \eqref{importante} and the previous observation, using Propositions \ref{sobolev-ineq}-\ref{compact}, following verbatim the proof of \cite[Theorem 1.3]{FelmerQuaasTan}  based on variational method together with concentration compactness principle we conclude the next existence result.
\begin{theorem}\label{existencia1}
If
\begin{itemize}
\item [i)] $-2s<\alpha$ when $1/2\leq s<1,$ or
\item [ii)] $-2s<\alpha<\frac{2s(N-1)}{1-2s}$ when $0<s<1/2,$
\end{itemize}
then the problem 
\begin{equation}\label{eq-uve-buena}
\mathcal{T} \bar{v} + \mathcal{A}_{s,N} \bar{v}= \bar{v}^{p^*_{\alpha,s}} \mbox{ in $\mathbb{R}$,}
\end{equation}
where $\mathcal{T}$ and $K$ where given in \eqref{laT} and \eqref{kernel} respectively, has a nonnegative variational solution. Moreover, since $f(t)=t^{p^*_{\alpha,s}}$ is a H\"older function, the solution is classical and, in fact, is positive.
\end{theorem}

The regularity is obtained again by a bootstrapping argument  based on \cite[Proposition 3.10 and Proposition 3.11]{Manuel} (see also \cite{schauder}) and on a $L^\infty$ bound of Proposition \ref{decaimiento} below.

As we commented before, the condition ii) of the previous theorem is trivially true for every $0<s<1/2$ and $-2s<\alpha<0$ as we expected. Moreover we emphasize that the previous result imply the existence of positive and radially symmetric solutions of the critical non local H\'enon equation even when $\alpha>0$, {that, as far as we know, it was not proved until the date.}
\subsection{Qualitative properties of solutions. The ``bubble". }
We prove now a qualitative property of the solutions of \eqref{eq-uve-buena} that will be the key step to find, later, the shape of the ``bubble" for the nonlocal H\'enon equation. More precisely we want to obtain that the solutions decay to zero in (plus and minus) infinity. Before going to the statement and proof of this result, we observe that reverse of the operator $\mathcal{T}$ seems not simple at all so we cannot follows closely the proof of \cite[Theorem 3.4]{FelmerQuaasTan} in order to get the desired qualitative property. Thus, instead of use this well known approach we will prove directly that the solutions are bounded and regular enough to conclude the desired decay.
\begin{prop}\label{decaimiento}
If $u\in H^s_{K}(\mathbb{R})$ is a positive variational solution of \eqref{eq-uve-buena} given by Theorem \ref{existencia1} then $u\in L^\infty(\mathbb{R})$ and $u(t)\to 0$ when $|t|\to \infty$.
\end{prop}
\begin{proof}
The key point of the proof is that, since for every convex function $\varphi$,
\begin{equation}\label{convexa}
\mathcal{T}(\varphi (u))\leq \varphi'(u)\mathcal{T}(u),
\end{equation}
we can adapt the ideas done in \cite[Proposition 2.2]{Bego} 
for the non bounded space $\mathbb{R}$ in order to obtain that 
\begin{equation}\label{bounded}
u\in L^{\infty}(\mathbb{R}).
\end{equation}

In fact let us define, for $\beta \ge 1$ and $M>0$ large,
$$
\varphi(x)=\varphi_{M,\beta}(x)=\left\{
\begin{array}{ll}
0, \quad&{\rm if} \quad t\leq0\\[2mm]
x^\beta, \quad&{\rm if} \quad 0<x<M\\[2mm]
\beta M^{\beta-1} (x-M)+M^\beta, \quad&{\rm if} \quad x\ge M.
\end{array}\right.
$$
Since $\varphi$ is Lipschitz, with constant $K=\beta M^{\beta-1} $, it is clear that $\varphi(u) \in H^s_{K}(\mathbb{R})$. Moreover by \eqref{convexa}, since $\varphi(u)\varphi'(u)\geq 0$ and $u>0$,
$$ \int_\mathbb{R} \varphi(u)\mathcal{T}(\varphi (u))\, d\tau\leq  \int_\mathbb{R}  \varphi(u)\varphi'(u) u^{p^*_{\alpha,s}}. $$
Using now that $u\varphi'(u)\leq \beta\varphi(u)$ by Proposition \ref{sobolev-ineq} and the fact that $p^*_{\alpha,s}<2^*_s-1$, the above estimate becomes
\begin{equation} \label{two}
\left(\int_\mathbb{R} \left(\varphi(u)\right)^{2^*_s}\right)^{\frac{2}{{2^*_s}}}\le C\, \beta \int_\mathbb{R} \left(\varphi(u)\right)^{2}u^{2^*_{s}-2}.
\end{equation}
Since $\beta\geq 1$ and $\varphi(u)$ is linear when $u\geq M$, it can be checked that both sides of (\ref{two}) are finite.

\
Let $\beta_1$ be such that $2\beta_1=2^*_{s}+1$ and  $R$ large to be determined later. Then, H\"older's inequality with $p=(2\beta_1-1)/2=2^*_{s}/2$ and $p'=2^*_{s}/(2^*_{s}-2)$ gives
\begin{eqnarray*}
\int_\mathbb{R} \left(\varphi(u)\right)^{2}u^{2^*_{s}-2}&\leq& \int_{\{u\le R\}} \frac{\left(\varphi(u)\right)^{2}}{u}R^{\, 2^*_{s}-1}\\
&+&\left(\int_\mathbb{R} \left(\varphi(u)\right)^{2^*_{s}}\right)^{\frac{2}{{2^*_{s}}}} \left(\int_{\{u>R\}} u^{2^*_{s}}\right)^{\frac{2^*_{s}-2}{2^*_{s}}}.
\end{eqnarray*}
By the Monotone Convergence Theorem, we may take $R$ so that
$$
\left(\int_{\{u>R\}} u^{2^*_{s}}\right)^{\frac{2^*_{s}-2}{2^*_{s}}}\le \frac 1{2 \, C\, \beta_1}.
$$
In this way, the second term above is absorbed by the left hand side of (\ref{two}) to get
$$
\left(\int_\mathbb{R} \left(\varphi(u)\right)^{2^*_s}\right)^{2/{2^*_s}} \le 2\, C\, \beta_1 \,
\left(\int_{\{u\le R\}} \frac{\left(\varphi(u)\right)^{2}}{u}R^{2^*_s-1} \right).
$$
Letting $M\rightarrow \infty$ in the left hand side, it follows that 
\begin{equation*}
\left(\int_\mathbb{R} u^{2^*_s\beta_1}\right)^{2/{2^*_s}} \le 2\, C\, \beta_1 \,
\left(R^{2^*_s-1} \int_\mathbb{R} u^{2^*_s}\right)<\infty.
\end{equation*}
Wich implies that $u\in L^{\beta_1 \, 2^*_{s}}(\mathbb{R})$. Let us now consider $\beta>\beta_1$. Going back to inequality (\ref{two}) and using, as before,  that $\varphi_{M,\beta}(u)\leq u^{\beta}$ in the right hand side and taking $M\rightarrow \infty$ in the left hand side we obtain
$$
\left(\int_\mathbb{R} u^{2^*_s\beta}\right)^{2/{2^*_s}} \le C\, \beta \,
\left(\int_\mathbb{R} u^{2\beta+2^*_s-2} \right).
$$
where $C>0$ is independent of $\beta$. Thus
\begin{equation}\label{c}
\left(\int_\mathbb{R} u^{2^*_s\beta}\right)^{\frac{1}{2^*_s(\beta-1)}}  \le (C\beta)^{\frac{1}{2(\beta-1)}} \left(\int_\mathbb{R} u^{2\beta+2^*_s-2} \right)^{\frac{1}{2(\beta-1)}}
\end{equation}
Defining now the same iterative process as in \cite{Bego} we get \eqref{bounded}. In fact let define $\beta_{m+1}$, $m\geq 1$ such that
so that 
$$2\beta_{m+1}+ 2^*_s-2=2^*_s\beta_m.$$Therefore
$$\beta_{m+1}-1=\left(\frac {2^*_s} 2\right)^m(\beta_1-1),\, m\geq 1$$
and replacing it in \eqref{c} we get
$$
\left(\int_\mathbb{R} u^{2^*_s\beta_{m+1}}\right)^{\frac{1}{2^*_s(\beta_{m+1}-1)}}  \le (C\beta_{m+1})^{\frac{1}{2(\beta_{m+1}-1)}} \left(\int_\mathbb{R}u^{2^*_s\beta_m} \right)^{\frac{1}{2^*_s(\beta_m-1)}},
$$
Then, defining for $m\geq 1$
$$A_m:=\left(\int_\mathbb{R} u^{2^*_s\beta_m}\right)^{\frac{1}{2^*_s(\beta_m-1)}}\mbox{ and } C_{m+1}:=C\beta_{m+1},$$
using a limiting argument, we conclude that there exists $C_0>0$, independent of $m>1$, such that
$$A_{m+1}\leq \prod_{k=2}^{m+1}{C_{k}^{\frac{1}{2(\beta_k-1)}}}A_1\leq C_0A_1,$$
wich implies that $\|u\|_{L^\infty(\mathbb{R})}\leq  C_0A_1.$

\
Once we have proved \eqref{bounded} it is possible to apply \cite[Theorem 1.1]{Kas} (see \cite[Proposition 3.10]{Manuel}) to obtain the H\"older regularity of the solution. More precisely we deduce that there exists $0<\gamma<1$ and $c>0$, depending on $s$, such that for every $R>0$
\begin{equation}\label{holder}
\frac{|u(\tau)-u(\kappa)|}{|\tau-\kappa|^{\gamma}}\leq \frac{c}{R^{\gamma}}\|u\|_{L^{\infty}(\mathbb{R})},\quad |\tau|\leq R,\, |\kappa|\leq R.
\end{equation}

Namely, recovering $\mathbb{R}$ with arbitrary balls $[-R, R]$, by \eqref{bounded} and \eqref{holder} we have obtained that $u\in L^{\infty}(\mathbb{R})\cap \mathcal{C}^{\gamma}(\mathbb{R})\cap L^{2^*_{s}}(\mathbb{R})$ which implies that $u(t)\to 0$ when $|t|\to \infty$ completing the proof.
\end{proof}

We are able now to obtain good bounds for the solutions of \eqref{eq-uve-buena}, that is, the kind of estimate given in c) of Theorem \ref{existencia}. In fact, if we prove that the solutions of \eqref{eq-uve-buena} satisfy $\bar{v}(t)\leq Ce^{-\left(\frac{N-2s}{2}\right)|t|},$ $t\in\mathbb{R}$,
undoing the Emden Folwer change of variable, we will get
\begin{equation}\label{star}
v(r)\leq C\left\{
\begin{aligned}
r^{\frac{N-2s}{2}} &&  r\to 0,\\
r^{-\frac{N-2s}{2}}&& r\to \infty.
\end{aligned}\right.
\end{equation}
Therefore by \eqref{dobleuve}, this will imply that
\begin{equation}\label{collar}
r^{-\frac{N-2s}{2}}v(r)\leq C\left\{
\begin{aligned}
1&&  r\to 0,\\
r^{-(N-2s)}&& r\to \infty,
\end{aligned}\right.
\end{equation}
that is, we get a bound for the solution of \eqref{Henon}.
Thus our next objective is the following
\begin{prop}\label{46}
If $u\in H^s_{K}(\mathbb{R})$ is a nonnegative variational solution of \eqref{eq-uve-buena} given by Theorem \ref{existencia1} then there exists $C>0$ such that 
\begin{equation}\label{cota}
u(t)\leq Ce^{-\left(\frac{N-2s}{2}\right)|t|},\quad t\in\mathbb{R}. 
\end{equation}
\end{prop}

\begin{proof}
First of all we notice that by \cite[Lemma 3.1, Corollary 3.1 and Remark 3.1]{felmer}, see also Remark \ref{extranew}, there exists a strictly concave function $c$ define on $(-N,2s)$ such that
\begin{equation}\label{blanca}
(-\Delta)^s r^{\mu}= c(\mu) r^{\mu-2s},\quad -N<\mu<2s.
\end{equation}
The function $c(\mu)$ is positive in $(-N+2s,0)$, has two zeros $c(0)=c(-N+2s)=0$ and satisfies $\displaystyle \lim_{\mu\to-N^{+}} c(\mu)=\lim_{\mu\to 2s^{-}} c(\mu)=-\infty$. Thus, considering $v(r):=r^{\mu-\beta}$ in \eqref{eq-dobleuve}, from \eqref{blanca}, it follows that
$$\mathcal{L} r^{\mu-\beta} + \mathcal{A}_{s,N} r^{\mu-\beta}= c(\mu) r^{\mu-\beta},\quad r>0,\quad  -N<\mu<2s,$$
where $\beta$ was given in \eqref{beta}. Doing the Emden Fowler change of variable, the previous identity is equivalent to 
\begin{equation}\label{auxiliar}
\mathcal{T} v_{\mu}(t) + \mathcal{A}_{s,N} v_{\mu}(t)= c(\mu) v_{\mu}(t), \mbox{with $v_{\mu}(t)=e^{t\left(\mu+\frac{N-2s}{2}\right)}$,}
\end{equation}
for every $t\in\mathbb{R},$ and $-N<\mu<2s$. Let fix now $R>0$ big enough and let us consider $u\geq 0$ a weak solution of \eqref{eq-uve-buena}. By Theorem \ref{existencia1} we know that this solution is classical and therefore we can evaluate the operator pointwise. Using now Proposition \ref{decaimiento} there exists $\delta(R)>0$ such that
$$\mathcal{T} u(t) + \mathcal{A}_{s,N} u(t)=u^{p^*_{\alpha,s}}(t)\leq \delta u(t),\quad |t|>R.$$
Since for some $-N<\mu_1<\mu_2<0\, (<2s)$ and the properties of the function $c(\mu)$, it is true that $c(\mu_i)=\delta$, $i=1,\, 2,$ using comparison twice and the fact that $u$ is bounded, from \eqref{auxiliar} we get that there exists $C>0$ such that 
$$u(t)\leq C_{i}e^{t\left(\mu_{i}+\frac{N-2s}{2}\right)_{i=1,\, 2}},\quad\mbox{for some $\mu_i$ and $|t|>R$.}$$
More precisely, writing
$$\mu_1:=(-N+2s)+\varepsilon_1\mbox{ and }\mu_2:=-\varepsilon_2,$$
for some positive $\varepsilon_1(\delta),\, \varepsilon_2(\delta)$ smaller than $N$, we get 
\begin{equation}\label{trapera}
u(t)\leq C_{1}e^{t\left(-\frac{N-2s}{2}+\varepsilon_{1}\right)}\mbox{ and } u(t)\leq C_{2}e^{t\left(\frac{N-2s}{2}-\varepsilon_{2}\right)},\quad |t|>R.
\end{equation}
Without loss of generality we prove \eqref{cota} for the case $t>0$ and the proof can be easily adapted for $t\leq 0$ (see Remark \ref{extranew} below). 
We define now
$$\widetilde{\mu}_1:=\left(\varepsilon-\frac{N-2s}{2}\right)p^*_{\alpha,s}> \left(\mu_1+\frac{N-2s}{2}\right)p^*_{\alpha,s}\left(>\mu_1\right),$$
with $\varepsilon>\varepsilon_1$ that will be chosen later (see \eqref{el-epsilon}). Since by \eqref{trapera} it is clear that
\begin{equation}\label{gripe}
\mathcal{T} u(t) + \mathcal{A}_{s,N} u(t)\leq e^{t\, \widetilde{\mu}_1},\quad |t|>R,
\end{equation}
our next objetive is to find a suitable function $\widetilde{u}_1$ satisfying \eqref{cota} and 
\begin{equation}\label{laec}
\mathcal{T}\, \widetilde{u}_1(t) + \mathcal{A}_{s,N} \widetilde{u}_1=e^{t\, \widetilde{\mu}_1}, \quad |t|>R,
\end{equation}
that will allow us to apply a comparison principle and obtain the same bounded estimate \eqref{cota} for the function $u$. For that, on the first hand, we notice that by \eqref{auxiliar}
$$\mathcal{T} e^{t\, \widetilde{\mu}_1} + \mathcal{A}_{s,N} e^{t\, \widetilde{\mu}_1}= c\left(\widetilde{\mu}_1-\frac{N-2s}{2}\right) e^{t\, \widetilde{\mu}_1},$$
as long as 
$$-N<\widetilde{\mu}_1-\frac{N-2s}{2}<2s.$$
Thus, for some $\varepsilon>0$, chosen later (see \eqref{el-epsilon}), the function
$$\widetilde{v}_1(t):=\frac{1}{c\left(\widetilde{\mu}_1-\frac{N-2s}{2}\right)}  e^{t\, \widetilde{\mu}_1},\quad t\in\mathbb{R},$$
is a particular solution of \eqref{laec}. On the other hand we know that, for every $M>0$,
\begin{equation}\label{thisisnew}
\widetilde{v}_{2,M}(t):= M e^{-t\left(\frac{N-2s}{2}\right)},\quad  t\in\mathbb{R},
\end{equation}
is a solution of the homogeneous equation $\mathcal{T}\cdot + \mathcal{A}_{s,N} \cdot =0$ in $\mathbb{R}$. We choose now
\begin{equation}\label{el-epsilon}
\max\left\{\varepsilon_1,\, \frac{\alpha(N-2s)}{N+2s+2\alpha}\right\}<\varepsilon<\frac{(\alpha+2s)(N-2s)}{N+2s+2\alpha}\, (<N),
\end{equation}
that in particular implies 
\begin{equation}\label{cesar}
-N<\widetilde{\mu}_1-\frac{N-2s}{2}<-N+2s\, (<2s),
\end{equation}
and also 
$$c\left(\widetilde{\mu}_1-\frac{N-2s}{2}\right)<0.$$
Therefore, for every $M>0$, the function
$$\widetilde{u}_{1,M}(t):=\widetilde{v}_{2,M}(t)+ \widetilde{v}_1(t),\quad  t\in\mathbb{R},$$
satisfies \eqref{laec}. Moreover choosing $M>0$ such that, by Proposition \ref{decaimiento},  $u(t)\leq \widetilde{u}_{1,M}(t)$ for $|t|<R$, by  \eqref{gripe} and comparison, it follows in particular that
$$u(t)\leq \widetilde{u}_{1,M}\leq M e^{-t\left(\frac{N-2s}{2}\right)}\mbox{ when $t>0$,}$$
as wanted.
\end{proof}

\begin{remark}\label{extranew}
\begin{itemize}
\item[i)]{To obtained the desirable bound for the solutions of \eqref{eq-uve-buena} in the case $t\leq 0$, the proof follows as the one done before for $t>0$ by using now the fact that $u(t)\leq e^{-t(\varepsilon+\beta)}$ for some $\varepsilon>\varepsilon_2$ (see \eqref{trapera}) and considering $\widetilde{v}_{2,M}(t):= M e^{-t\beta}$  instead of \eqref{thisisnew} with $\beta$ defined in \eqref{beta} and some $M>0$. We highlight that the previous function also satisfies $\mathcal{T}\cdot + \mathcal{A}_{s,N} \cdot =0$ in $\mathbb{R}$ because the constants function are, trivially, $s$-harmonic in all the space. The conclusion follows because we get 
$$u(t)\leq M e^{-t\beta}+\frac{1}{c\left(-\widetilde{\mu}_1-\frac{N-2s}{2}\right)}  e^{-t\, \widetilde{\mu}_1},\quad t\in\mathbb{R},$$
with $\widetilde{\mu}_1$ verifying \eqref{cesar}.}

\item[ii)] We show up that, from the proof presented above, it can be deduced that the estimate \eqref{cota} is also true for every regular solution of \eqref{eq-uve-buena}, that is, not only for the one given by Theorem \ref{existencia1}.

\item[iii)] We notice here that in \cite{Rupert} the explicit value
$$c(\mu)=4^{s}\frac{\Gamma\left(\frac{-\mu+2s}{2}\right)\Gamma\left(\frac{N+\mu}{2}\right)}{\Gamma\left(\frac{N+\mu-2s}{2}\right)\Gamma\left(-\frac{\mu}{2}\right)},$$
can be deduce for the range $-N+2s<\mu<0$. From the previous expression is clear that
$$c(\beta)=4^{s}\frac{\Gamma^2\left(\frac{N+2s}{4}\right)}{\Gamma^2\left(\frac{N-2s}{4}\right)}>0,$$
with $\beta\in (-N+2s,0)$ defined in \eqref{beta}. Thus, since $\mathcal {A}_{s,N}=c(\beta),$
where $\mathcal {A}_{s,N}$ was given in \eqref{a}, the previous observation is an alternative way to prove Proposition \ref{prop-kernel} i).
\end{itemize}
\end{remark}

We are now in position to prove our second main Theorem 
\begin{proof}[Proof Theorem 1.2] 
Let us consider 
$$u(r):=r^{-\frac{N-2s}{2}}v(r),$$
where $v(r)$ was obtained in Theorem \ref{existencia1}.

a) The existence part is direct if $\alpha>0$ since the previous function $u(r)$ is, clearly, a classical solution of \eqref{Henon} for $p=p^{*}_{\alpha,s}$.

{b) If $0>\alpha>-2s$ the solution are not regular due to the singularity at zero, so we need to argue in a different way. 
First of all we prove that $u$ is a {\it strong} solution of \eqref{Henon} with $p=p^{*}_{\alpha,s}$ where, here, the notion of {\it strong} solutions is the one that appears in \cite[Definition 2.4]{Bras1}, that is, the functions for which the principal 
value converges in $L_{loc}^1(\mathbb{R}^{N})$. Let us consider $x\in K$, where $K$ is a compact subset. We define
$$g_\epsilon(x):= a_{N,s}\int_{\mathbb{R}^{N}\setminus B_\varepsilon} \frac{u(x)-u(y)}{|x-y|^{N+2s}} dy,$$
and
$$g(x):=u(x)^{p^*_{\alpha,s}}|x|^\alpha.$$
By the fact that $u\in C^{1,\gamma}(\mathbb{R}^{N}\setminus\{0\})\cap L^{\infty}(\mathbb{R}^{N})$, $\gamma>2s-1$, (see Lemma \ref{weak} and \eqref{collar}), it can be proved that 
$$\mbox{$|g_\epsilon(x)|\leq C(K,\|u\|_{L^{\infty}}, \|u\|_{C^{1,\gamma}}):=h$ for every $\epsilon>0$ and $x\in K$.}$$
Thus, since $h \in L^1_{loc}(\mathbb{R}^N)$ and $g_{\epsilon}(x)\to g(x)$, $x\in K$, we can apply the Dominate Convergent theorem to get that $g_\epsilon \to g$  in $L^1_{loc}(\mathbb{R}^{N})$ obtaining that $u$ is a {\it strong} solution of \eqref{Henon} when $p=p^{*}_{\alpha,s}.$}  

{Finally we notice that  $u (-\Delta)^s u=u^{p^*_{\alpha,s}+1} |x|^\alpha\in L^1(\mathbb{R}^{N})$ because $u$ is fast decay (see \eqref{collar}) and $\alpha>-2s$. Then, by the Plancherell identity and \cite[Proposition 3.4]{guida} we get
$$\int_{\mathbb{R}^{N}} \frac{(u(x)-u(y))^2}{|x-y|^{N+2s}} dy<\infty.$$
Since $u$ is locally bounded we have that $u\in H^s_{loc}(\mathbb{R}^{N})$ {($u\in H^s(\mathbb{R}^{N})$  in the case $N\geq 4$)} and, therefore, by \cite[Corollary 2.7]{{Bras1}} we conclude that $u$ is a also a {\it weak} solution of \eqref{Henon} with $p=p^*_{\alpha,s}$.}
\\

c) For the lower bound of $u$. We observe that, since $(-\Delta)^s u\geq 0$ in $\mathbb{R}^{N}\setminus \{0\}$,
by  \cite[Proposition 3.4 and Proposition 3.5]{Bras}  there exists a sub-solution with the desired lower bound. Thus by comparison we get that the estimate is also true for $u$, that is, $u(x)\geq c_1 |x|^{-N+2s}.$
\\Using the decay estimate obtained for $v$ in Proposition \ref{46} clearly we also obtain the upper bound of c) (see \eqref{collar}).

\end{proof}

{\bf Acknowledgements.}
B. B. was partially supported by AEI grant MTM2016-80474-P.
A. Q. was partially supported by Fondecyt Grant No. 1151180 and Programa Basal, CMM. U. de Chile


\end{document}